	\documentclass[12pt,a4paper,oneside]{article}

	\usepackage{amsmath, amsthm, amsfonts, amssymb}
	\usepackage[all]{xy}
	\usepackage{enumerate}
	\usepackage[latin1]{inputenc}

	\swapnumbers

	\theoremstyle{plain}
		\newtheorem{theorem}	  {Theorem}[subsection]
		\newtheorem{proposition}[theorem]{Proposition}
		\newtheorem{corollary}	[theorem]{Corollary}
		\newtheorem{lemma}			[theorem]{Lemma}
	\theoremstyle{remark}
		
		\newtheorem*{example}		{Example}
		\newtheorem*{examples}		{Examples}

		\def\cat	{{\rm Cat}}
		\def\fib	{{\rm Fib}}
		\def\ab		{{\rm Ab}}
		\def\cliv	{{\rm Cliv}}
		\def\esc	{{\rm Esc}}
		\def\set	{{\rm Set}}

		\def\top	{{\rm Top}}
		\def\sset	{{\rm SSet}}
		\def\bsset	{{\rm bSSet}}

		\def\xto{\xrightarrow}
		\def\then{\Rightarrow}
		
		\def\ob{{\rm ob}}
		\def\fl{{\rm fl}}
		\def\id{{\rm id}}
		
		\def\dom{{dom}}
		\def\cod{{cod}}
		\def\colim{{colim}}

		\def\gs{\geqslant}
		
		\def\leq{\leqslant}
		
		\def\u{\underline}
		
		\def\Z{\mathbb Z}

		\def\ss{\square}

    \def\title#1{\noindent{\bf\LARGE{#1}} \bigskip \thispagestyle{plain}}
    \def\author#1{\noindent{\sc #1}\smallskip}
    \def\address#1{\noindent #1}

\begin{document}

\title{On the homotopy type of a cofibred \\category}

\author{Matias L. del Hoyo}

\address{Departamento de Matem\'atica\\FCEyN, Universidad de Buenos Aires\\Buenos Aires, Argentina.}

\thispagestyle{empty}


\begin{abstract}
In this paper we describe two ways on which (co)fibred categories give rise to bisimplicial sets.
The {\em fibred nerve} is a natural extension of Segal's classical nerve of a category, and it constitutes an alternative simplicial description of the homotopy type of the total category. If the fibration is splitting, then one can construct the {\em cleaved nerve}, a smaller variant which emerges from a closed cleavage.
We interpret some classical theorems by Thomason and Quillen in terms of our constructions, and use the fibred and cleaved nerve to establish new results on homotopy and homology of small categories.
\end{abstract}

%



{\small {\bf 2000 MSC:}
18D30; 
18G30; 
55U35. 

}

\section*{Introduction}

The classifying space functor associates to every small category $C$ a topological space $BC$, namely the geometric realization of its nerve \cite{segal}.
The classical homotopy theory of categories is lifted from spaces by using this functor.
For instance, a {\em weak equivalence} between small categories is a map $f:C\to C'$ such that $Bf$ is a homotopy equivalence.

A fundamental fact concerning this construction is that
 for every space $X$ there is a small category $C$ such that $X$ and $BC$ have the same weak homotopy type (cf. \cite[VI,3.3.1]{illusie}, see also \cite{del hoyo}).
This way small categories constitute models for homotopy types, and one seeks to characterize the discrete invariants of $X$ in terms of its underlying category $C$.

It is natural to expect that a small category $C$ endowed with extra structure would give rise to a space $BC$ equipped with some additional data. That is our motivation for introducing the {\em fibred nerve} and the {\em cleaved nerve}. These are bisimplicial sets with the homotopy type of the total category of a Grothendieck fibration, and constitute combinatorial descriptions that preserve in some sense the fibred structure.

By a Grothendieck fibration, or just a {\em fibration}, we mean what is usually called a {\em cofibred category}.
We adopt this terminology for simplicity, and to emphasize the analogy with the topological case.
Other notions of fibrations between small categories have been studied, for instance, in \cite{minian,evrard}.

Grothendieck fibrations have played an important role in homotopy theory. Among others, they were used by Thomason to describe homotopy colimits of small categories \cite{thomason}, and Quillen's Theorems A and B -- that lead to long exact sequences of higher K-theory groups -- may be stated in terms of Grothendieck fibrations \cite{quillen}.
We believe that the nerve constructions studied here will help in further applications, such as explicit constructions of $K(G,n)$ categories and Postnikov towers in $\cat$.

\subsubsection*{Organization}

Section 1 deals with preliminaries. 
We fix some notations and recall some results about the classifying space functor and a key proposition on simplicial sets (\ref{levelwise}). The reader is referred to \cite{quillen} for an introduction to homotopy of small categories, and to \cite{gj} for a comprehensive treatment of bisimplical objects.

The principal reference on Grothendieck fibrations is \cite[VI]{sga1}. A more recent one is \cite{borceux}. In section 2 we set the definitions, recall some facts about fibrations and develop some others which will be needed later, such as the correspondence \ref{s sigma}.

In section 3 we introduce both the fiber and the cleaved nerve, in the same fashion as the classical nerve is defined. We establish some fundamental facts (cf. \ref{fiberwise}, \ref{cleaved fiberwise}) and prove that for a splitting fibration the two constructions yield the same homotopy type (cf. \ref{fibrado clivado}).

We prove that the fibred nerve is homotopy equivalent to the classic nerve in section 4 (cf. \ref{thm2}). From these we derive the original and the relative versions of Quillen's theorem A. In addition, we show how to recover the classic nerve of a splitting fibration using the codiagonal construction over a bisimplicial set (cf. \ref{codiag3}).

The last section summarizes applications and relations between the fibred nerve and some other constructions.
\begin{itemize}	  \itemsep=0ex
\item The cleaved nerve and Bousfield-Kan construction for homotopy colimits are related in \ref{hc}. We derive Thomason's theorem on homotopy colimits of small categories as a consequence.
\item We develop a Leray-Serre style spectral sequence (cf. \ref{spectral sequence}) relating the homology groups of the base, the fibers and the total category. We deduce as a corollary a homology version of Quillen's Theorem A (cf. \ref{homology2}).
\item We introduce Quillen fibrations, which are families of categories with the same homotopy type, and show that Quillen's Theorem B might be interpreted as the following conceptual fact: the fibred classifying space functor maps Quillen fibrations into quasifibrations (\ref{prop}).
\item Finally, we associate to a category endowed with a group action a splitting fibration, and prove that its cleaved nerve is a twisted cartesian product as defined in \cite{may}.
\end{itemize}

\subsubsection*{Acknowledgements}

I would like to thank Gabriel Minian, my advisor. His several suggestions and remarks were essential in the development and revision of this work. I also thank to Fernando Cukierman and Eduardo Dubuc for many stimulating talks. Lastly, I thank to CONICET for the finantial support.

\section{Preliminaries}

We denote by $\cat$, $\sset$ and $\top$ the categories of small categories, simplicial sets and topological spaces, respectively.
If $C$ is a small category, then we denote by $\ob(C)$ its set of objects and by $\fl(C)$ its set of arrows.
As usual we denote the category of (non-empty) finite ordinals by $\Delta$ and by ${\bf\u n}=\{0,\dots,n\}$, the ordinal with $n+1$ elements.
We write $I$ for the simplicial set represented by ${\bf\u 1}$. Sometimes ${\bf\u n}$ will be regarded as a category in the usual way.

\subsection{About homotopy of small categories}

Given $C$ a small category, its {\em nerve} $NC$ is the simplicial set whose $n$-simplices are the chains
$$\u c=(c_0\to c_1\to...\to c_n)$$
of $n$ composable arrows in $C$, and its {\em classifying space} $BC$ is the geometric realization of its nerve, namely $BC=|NC|$. It is a CW-complex with one $n$-cell for each chain of $n$ composable arrows in $C$ which does not involve an identity \cite{segal}.

A functor $f:C\to C'$ in $\cat$ is a {\em weak equivalence} if $Bf$ is a homotopy equivalence in $\top$, and a small category $C$ is {\em contractible} if $BC$ is so. From the homeomorphism $B(C\times I)\cong BC\times BI$ it follows that a functor $C\times I\to C'$ induces a continuous map $BC\times[0,1]\to BC'$ and therefore
a natural transformation $h:f\then g:C\to C'$ yields a homotopy $Bh:Bf\then Bg:BC\to BC'$. This leads to the following results \cite{quillen}.

\begin{lemma}\label{adjoint}
If a functor admits an adjoint, then it is a weak equivalence.
\end{lemma}

\begin{lemma}\label{contractible}
A category having an initial or final object is contractible.
\end{lemma}

It is well known that given a commutative triangle in $\top$, if two of the three arrows involved are weak homotopy equivalences, then so does the third. It follows immediately that the same statement holds for weak equivalences in $\sset$ and weak equivalences in $\cat$. We will refer to this fact as the {\em 3-for-2-property}.

\subsection{About bisimplicial sets}

The nerve of a category is a simplicial set. We shall extend this concept by constructing the fibred nerve of a fibration, which is a bisimplicial set. A {\em bisimplicial set} is a functor $K:\Delta^{\circ}\times\Delta^{\circ}\to\set$, where $\Delta^{\circ}$ denotes the opposite category of $\Delta$. 
A bisimplicial set $K$ can be regarded as a family of sets $\{K_{m,n}\}_{m,n\geq0}$ equipped with horizontal and vertical faces and degeneracies operators satisfying the simplicial identities,
and such that the horizontal and vertical operators commute \cite{gj}.
We denote by $\bsset$ the category of bisimplicial sets and morphisms between them.

A bisimplicial set is the same as a simplicial object in $\sset$.
Given $K$ a bisimplicial set, let $\{m\mapsto K_{m,n}\}$ be the $n$-th {\em vertical} simplicial set, which is obtained from $K$ by setting the second coordinate equal to $n$. The $m$-th {\em horizontal} simplicial set $\{n\mapsto K_{m,n}\}$ is defined analogously. We denote by $d(K)$ the {\em diagonal} of $K$, namely the simplicial set which is the composition of $K$ with the diagonal functor $\Delta^{\circ}\to\Delta^{\circ}\times\Delta^{\circ}$.

We define the {\em geometric realization} of $K$ as the space $|d(K)|$, which is naturally homeomorphic to 
the spaces obtained by first realizing on one direction and then on the other \cite[p.10]{quillen}.
$$|n\mapsto|m\mapsto K_{m,n}|| \cong |d(K)| \cong |m\mapsto|n\mapsto K_{m,n}||$$
If $f:K\to L$ is a map of bisimplicial sets, then we say that it is a {\em weak equivalence} if its geometric realization $f_*:|d(K)|\to|d(L)|$ is a homotopy equivalence. The following is a very useful criterion to establish when a map is a weak equivalence
(see e.g. \cite[XII,2.3]{bk} or \cite[IV,1.9]{gj}).

\begin{proposition}\label{levelwise}
Let $f:X\to Y$ be a map in $\bsset$ such that for all $n$ the induced map
$f_*:\{m\mapsto X_{m,n}\} \to \{m\mapsto Y_{m,n}\}$ is a weak equivalence in $\sset$. Then $f$ is a weak equivalence.
\end{proposition}

\section{Fibrations}

If $u:A\to B$ is a map between small categories, we say that $f\in\fl(A)$ is {\em over} $\phi\in\fl(B)$ if $u(f)=\phi$, and we say that $f\in\fl(A)$ is {\em over} $b\in\ob(B)$ if $u(f)=\id_b$.
Given $b\in\ob(B)$, the {\em fiber} $u_b$ is the subcategory of $A$ of arrows over $b$, and the {\em homotopy fiber} $u/b$ is the category whose objects are pairs $(a,\phi)$, $a\in\ob(A)$ and $\phi:u(a)\to b\in\fl(B)$, and whose arrows $f:(a,\phi)\to(a',\phi')$ are maps $f:a\to a'$ in $A$ such that $\phi'u(f)=\phi$.
By an abuse of notation we shall write $A_b$ and $A/b$ instead of $u_b$ and $u/b$. Note that there is a canonical fully faithful inclusion $A_b\to A/b$, defined by $a\mapsto (a,id_{u(a)})$.

\subsection{Basic definitions and examples}

Let $p:E\to B$ a map between small categories.
An arrow $f:e\to e'$ in $E$ is said to be {\em cartesian} if it satisfies the following universal property:
for all $g:e\to e''$ over $p(f)$ there is a unique $h:e'\to e''$ over $p(e')$ such that $hf=g$.
$$\xymatrix@R=15pt{
					e \ar[rd]_{\forall g} \ar[r]^f & e' \ar@{-->}[d]^{\exists!h}\\ 
					  & e''
				  \\
					p(e) \ar[r]^{p(f)}& p(e')}$$

A map $p:E\to B$ is a {\em prefibration} if for any $e$ object of $E$ and any $\phi:p(e)\to b$ arrow of $B$ there is a cartesian arrow $f:e\to e'$ over $\phi$. 
It is not hard to see that $p:E\to B$ is a prefibration if and only if the inclusion $E_b\to E/b$ of the actual fiber into the homotopy fiber admits a left adjoint for all objects $b$ in $B$. Therefore, if $p:E\to B$ is a prefibration  then the inclusion $E_b\to E/b$ is a weak equivalence for all $b$ (cf. \ref{adjoint}).

A prefibration $p:E\to B$ is called a {\em fibration} if cartesian arrows are closed under composition. We say that $B$ is the {\em base category} and that $E$ is the {\em total category} of the fibration.

\begin{examples}\

\begin{itemize}

\item The projection $\pi:F\times B\to B$ is a fibration, since the arrows $(\id,\phi)\in\fl(F\times B)$ are cartesian.

\item Given $B$ a small category and $F:B\to\cat$ a functor, the projection $F\rtimes B\to B$ is a fibration, whose fibers are the values of $F$. Here $F\rtimes B$ denotes the {\em Grothendieck construction} over $F$ (see e.g. section \ref{Grothendieck construction}).

\item We denote by $B^I$ the category of functors $I\to B$. Its objects are the arrows of $B$ and its maps $(u,v):f\to g$ are the commutative squares $vf=gu$ in $B$.
The functor $\cod:B^I\to B$ which assigns to each arrow its codomain is a fibration.
The fibers of $\cod$ are the {\em slice} categories $B/b$.

\item If $A\subset B$ is a {\em coideal} (cf. \cite{heggie}), then the inclusion $A\to B$ is a fibration, whose fibers are either $\emptyset$ or ${\rm pt}$, the final object of $\cat$.

\item Given $p:E\to B$, an isomorphism $f\in\fl(E)$ is always cartesian. Thus, a functor between groupoids that is onto on arrows is a fibration.

\end{itemize}
\end{examples}

The cartesian arrows in a fibration satisfy the following stronger universal property (cf. \cite{borceux}).

\begin{lemma}\label{upc}
Let $p:E\to B$ be a fibration and $f:e\to e'$ a cartesian arrow in $E$.
Given $g:e\to e''$ such that $p(g)=\phi p(f)$ for some $\phi:p(e')\to p(e'')$, there exists a unique arrow $h:e'\to e''$ over $\phi$ satisfying $hf=g$.
$$\xymatrix@R=10pt{
					e \ar[rrd]_{\forall g} \ar[r]^f & e' \ar@{-->}[dr]^{\exists! h}\\ 
					 &  & e''				  \\
					p(e) \ar[r]^{p(f)} & p(e') \ar[r]^{\phi} & p(e'')}$$
Moreover, $h$ is cartesian if and only if $g$ is so.
\end{lemma}

Given a prefibration $p:E\to B$, a {\em cleavage} $\Sigma$ is a choice of cartesian arrows. 
More precisely, a cleavage is a subset $\Sigma\subset \fl(E)$ whose elements are cartesian arrows and such that for all $e\in \ob(E)$ and $\phi:p(e)\to b \in \fl(B)$ there exists a unique arrow $\Sigma_{e,\phi}:e\to e'$ in $\Sigma$ over $\phi$. 

The cleavage $\Sigma$ is said to be {\em normal} if it contains the identities, and is said to be {\em closed} if it is closed under composition. Every prefibration admits a normal cleavage, but not every prefibration admits a closed one. A fibration which admits a closed cleavage is called a {\em splitting fibration}.

\begin{example}
Let $E,B$ be groups, regarded as categories with a single object, and let $p:E\to B$ be a map between them. Then every map of $E$ is cartesian as it is an isomorphism. It follows that $p$ is a fibration if and only if $p$ is an epimorphism of groups. A cleavage $\Sigma$ for $p$ is a set-theoretic section for $p$. The cleavage is normal if $\Sigma$ preserves the neutral element, and the cleavage is closed if it is a morphism of groups.
This example shows in particular that ``only a few'' fibrations are splitting.
\end{example}

From here on we will assume that all the cleavages are normal. The following lemma, whose proof is straight-forward, gives an alternative description of closed cleavages.

\begin{lemma}\label{coro}
A cleavage $\Sigma$ is closed if and only if $f\in\Sigma$ and $f'f\in\Sigma$ imply that $f'\in\Sigma$ for all pair $f,f'$ of composable arrows of $E$.
\end{lemma}

\medskip

Next we discuss two notions of morphism between fibrations, and describe the corresponding categories.

Given $\xi=(p:E\to B)$ and $\xi'=(p':E'\to B')$  fibrations, a {\em fibred map} $(f,g):\xi\to\xi'$ is a pair $f:E\to E'$, $g:B\to B'$ of maps in $\cat$ such that $f$ preserves cartesian arrows and  $p'f=gp$.
We denote by $\fib(\xi,\xi')$ the set of fibred maps $\xi\to\xi'$, and by
$\fib$ the category of fibrations and fibred maps between them. 

Now suppose that cleavages $\Sigma$ and $\Sigma'$ of $\xi$ and $\xi'$ are given.
A {\em cleaved map} $(f,g):(\xi,\Sigma)\to(\xi',\Sigma')$ is a fibred map $(f,g):\xi\to\xi'$ such that $f(\Sigma)\subset\Sigma'$.
By $\cliv((\xi,\Sigma),(\xi',\Sigma'))$ we mean the set of cleaved maps $(\xi,\Sigma)\to(\xi',\Sigma')$, and by $\cliv$ the category of pairs $(\xi,\Sigma)$ and cleaved maps.

Finally, we denote by $\esc$ the full subcategory of $\cliv$ whose objects are the pairs $(\xi,\Sigma)$ with $\Sigma$ a closed cleavage of $\xi$.

We have the following diagram, where the first is a full inclusion and the arrow $\cliv\to\fib$ is the forgetful functor $(\xi,\Sigma)\mapsto \xi$.
$$\esc\subset\cliv\to\fib\subset\cat^I$$


With the notations of above, we will say that $f:E\to E'$ is a {\em fibred map over $B$} if $B=B'$ and $(f,\id_B):\xi\to\xi'$ is a fibred map. A {\em cleaved map over $B$} is defined similarly.

\subsection{Fibration associated to a map}\label{decomposition}

Given $u:A\to B$ a map between small categories, we define the {\em mapping category} $E^u$ as the fiber product $A\times_{B}B^{I}$ over $u$ and $\dom:B^I\to B$ in $\cat$.
The objects of $E^u$ are pairs $(a, u(a)\to b)$, with $a$ an object of $A$ and $u(a)\to b$ an arrow of $B$, and the arrows are pairs $(f,g)$ which induce a commutative square in $B$.

The functor $u$ factors through $E^u$ as $\pi i$,
where $i$ is the inclusion $a\mapsto(a,\id_{u(a)})$, and $\pi$ is the projection $(a,u(a)\to b)\mapsto b$.
$$\xymatrix{ A \ar[r]^i  \ar@/_1pc/[rr]_u & E^u \ar[r]^{\pi}  & B}$$
The functor $i$ is fully faithful and admits a right adjoint, the retraction $r:E^u\to A$, which maps $(a,u(a)\to b)$ into $a$. This implies the following (cf. \ref{adjoint}).

\begin{lemma}\label{replacement}
The map $i:A\to E^u$ is a weak equivalence.
\end{lemma}

The functor $\pi$ is a fibration. 
The set $\Sigma^u\subset\fl(E^u)$ of arrows whose first coordinate is an identity
$$\Sigma^u=\{(\id_a,\phi):(a,u(a)\to b)\to(a,u(a)\to b')\},$$
is a closed cleavage for $\pi$, so it is a splitting fibration.
We say that $\pi:E^u\to B$ is the {\em fibration associated to} $u$, and we endow it with the cleavage $\Sigma^u$. Note that if $b$ is an object of $B$, then the fiber $E^u_b$ of $\pi$ is isomorphic to the homotopy fiber $A/b$ of $u$.

\medskip

Except in very special situations, the retraction $r:E^u\to A$ does not commute with the projections, namely $(r,\id_B)$ is not a map in $\cat^I$. We shall describe how to replace $r$ by others well-behaved retractions when the map $u$ is already a fibration.

Let $p:E\to B$ be a fibration, and let $\pi:E^p\to B$ be its associated fibration.
We say that a map $s:E^p\to E$ is {\em good} if $si=\id_E$, $ps=\pi$ and $s$ preserves cartesian arrows.
$$\xymatrix@C=10pt@R=15pt{E \ar[rr]^i \ar[dr]_p &  & E^p \ar@/^/[ll]^s \ar[dl]^{\pi}\\ & B}$$
If $s$ is good, then $s$ is a fibred map over $B$.

\begin{lemma}\label{good good}
A good map $s$ is a weak equivalence, and it induces a weak equivalence $s:E^p_b\to E_b$ for all object $b$ in $B$.
\end{lemma}
\begin{proof}
The first statement holds by \ref{replacement} since $s$ is left inverse to $i$. About the second, note that under the isomorphism $E^p_b\cong E/b$ the induced map $E^p_b\to E_b$ identifies with a left inverse to the inclusion $E_b\to E/b$, which is a weak equivalence indeed.
\end{proof}

\begin{proposition}\label{s sigma}
Given $p:E\to B$ a fibration, there is a 1-1 correspondence between (normal) cleavages of $E$ and good maps $s:E^p\to E$.
\end{proposition}
\begin{proof}
Let $s:E^p\to E$ be a good map. For each $e\in\ob(E)$ and $\phi:p(e)\to b\in\fl(B)$ the arrow  $(id_e,\phi):(e,\id:p(e)\to p(e))\to(e,\phi:p(e)\to b)$ is cartesian in $E^p$. Therefore, $s(id_e,\phi)$ is a cartesian arrow of $E$ over $\phi$ with domain $s(i(e))=e$. It follows that the family
$\Sigma=\{s(id_e,\phi)\}_{e,\phi}$
is a cleavage of $E$, and it is normal because $s(\id_e,\id_{p(e)})=s(i(\id_e))=\id_e$.

Conversely, if $\Sigma$ is a normal cleavage of $E$, then we shall construct a good map $s=s(\Sigma):E^p\to E$  as follows.
An object $(e,\phi:p(e)\to b)$ in $E^p$ is mapped by $s$ into the codomain of $\Sigma_{e,\phi}\in\ob(E)$.
An arrow $(\alpha,\beta):(e,\phi:p(e)\to b)\to(e',\phi':p(e')\to b')$ of $E^p$ is mapped by $s$ into the unique arrow over $\beta$ which makes the following diagram commutative.
$$\xymatrix@R=10pt@C=15pt{
e\ar[rr]^{\alpha}\ar[rd]_{\Sigma_{e,\phi}}& & e'\ar[rd]^{\Sigma_{e',\phi'}}&\\
&s(e)\ar@{-->}[rr]_{s(\alpha,\beta)}& &	s(e') \\
p(e)\ar[rr]^{p(\alpha)}	\ar[rd]_{\phi}&	& p(e')\ar[rd]^{\phi'}& \\
 &	b	\ar[rr]_{\beta}	&	&			b'}$$
The uniqueness of $s(\alpha,\beta)$ follows from \ref{upc}. It also follows from \ref{upc} that $s$ preserves cartesian arrows. As it respects identities and compositions, $s$ is indeed a functor, and $ps=\pi$ by construction.
The map $s$ defined this way is a retraction for $i:E\to E^p$ because $\Sigma$ is normal.

It is straightforward to check that these procedures are mutually inverse.
\end{proof}

If $E$ is endowed with a cleavage $\Sigma$ and $s:E^p\to E$ is a good map such that $s(\Sigma^u)\subset\Sigma$, then we say that $s$ is {\em very good}.
If $s$ is very good, then $s$ is a cleaved map over $B$.

\begin{corollary}\label{very good map}
If $s$ and $\Sigma$ are related as in \ref{s sigma}, then $\Sigma$ is closed if and only if the map $s$ is very good.
\end{corollary}
\begin{proof}
Let $\Sigma$ be a closed cleavage and $s$ its induced good map. If $(\id_e,\beta)$ is an arrow in $\Sigma^u$, then the diagram of above gives $s(\id_e,\beta)\Sigma_{e,\phi}=\Sigma_{e,\beta\phi}$. It follows from \ref{coro} that $s(\id_e,\beta)\in\Sigma$ and hence the map $s$ is very good.

On the other hand, given $\Sigma$ a cleavage which is not closed, by \ref{coro} one can find $f$ and $f'$ cartesian arrows of $E$ such that $f'=gf$ with $g\notin\Sigma$. Since $g=s(id,p(g))$ it follows that $s$ is not very good.
\end{proof}

\section{Bisimplicial sets from fibrations}

\subsection{Fibred nerve}\label{ultima}

For $m,n\gs 0$ let $\ss_{m,n}$ denotes the fibration $pr_2:{\bf\u m}\times{\bf\u n}\to {\bf\u n}$. These are the fibrations which play the role of simplices in $\fib$.  They define a covariant functor $\ss:\Delta\times\Delta\to \fib$. 

Given $\xi=(p:E\to B)$ a fibration, we define the {\em fibred nerve of $\xi$} as the bisimplicial set $N_f\xi$ whose $m,n$-simplices are given by
$$N_f\xi_{m,n}=\fib(\ss_{m,n},\xi).$$
We define the {\em fibred classifying space} $B_f\xi$ as the geometric realization $|d(N_f\xi)|$ of the fibred nerve.
These constructions are functorial.
For short, we shall write $N_fE$ and $B_fE$ instead of $N_f\xi$ and $B_f\xi$.

The fibred nerve extends the classical nerve in the sense that there exists a natural isomorphism 
$$d(N_f(\id_B))=d(N_fB)\cong NB.$$

A $m,n$-simplex of $N_fE$ consists of a pair $s=(s_0,s_1)$, where $s_0:{\bf\u m}\times{\bf\u n}\to E$ and $s_1:{\bf\u n}\to B$ are such that the induced square commutes.
We say that $s_1\in NB_n$ is the {\em base} of the simplex $s$, and that $s_0|_{pr_2^{-1}(0)}\in (NE_{b_0})_m$ is the {\em mast} of $s$. Of course, $s_0$ completely determines $s$.

We visualize $s$ as an array of arrows of $E$ going down and right. The horizontal arrows are cartesian and the vertical arrows are over identities.
$$\xymatrix@R=10pt@C=20pt{
e_{0,0}\ar[r]\ar[d] & e_{0,1}\ar[r]\ar[d] & \dots\ar[r] & e_{0,n}\ar[d] \\
e_{1,0}\ar[r]\ar[d] & e_{1,1}\ar[r]\ar[d] & \dots\ar[r] & e_{1,n}\ar[d] \\
\dots\ar[d]		   & \dots\ar[d]		& \dots		  & \dots\ar[d]  \\
e_{m,0}\ar[r]       & e_{m,1}\ar[r]		& \dots\ar[r] & e_{m,n}		 \\
b_0\ar[r]		   & b_1\ar[r]			& \dots\ar[r] & b_n			 }$$
Sometimes we will write $e_{i,j}^s$ to denote $s_0((i,j))$, and $e_{i,j}^s\to e_{i',j'}^s$ to denote $s_0((i,j)\to(i',j'))$.

The next technical result plays a key role hereafter. Fix $\u b\in NB_n$, and let $N_fE_{\u b}$ be the simplicial set whose simplices are those of $N_fE$ with base $\u b$, with faces and degeneracies in the vertical direction. 

\begin{lemma}\label{m}
The map $\mu:N_fE_{\u b}\to NE_{b_0}$ which assigns to each simplex $s$ its mast is a weak equivalence of simplicial sets.
\end{lemma}

\begin{proof}
We choose a cleavage $\Sigma$ and construct a homotopy inverse $\nu:NE_{b_0}\to N_fE_{\u b}$ for $\mu$ as follows. The map $\nu$ associates to a simplex $\u a$ the unique simplex $s=\nu(\u a)$ with mast $\u a$ and base $\u b$ and such that $e_{i,j}^s\to e_{i,j+1}^s\in\Sigma$ 
for all $i,j$. It is clear that $\mu\nu=\id$. We shall describe a simplicial homotopy
$h:N_fE_{\u b}\times I\to N_fE_{\u b}$
between $\nu\mu$ and $\id$, which induces a continuous homotopy $|h|$ and completes the proof.

We have that $(N_fE_{\u b}\times I)_m=(N_fE_{\u b})_m\times I_m$, and that $I_m=\{t:{\bf\u m}\to{\bf\u 1}\}$.
Given $(s,t)\in (N_fE_{\u b}\times I)_m$ we define $h(s,t)$ as the unique $m$-simplex of $N_fE_{\u b}$ with the same mast as $s$ and such that 
$$e_{i,j}^{h(s,t)}\to e_{i,j+1}^{h(s,t)}=
	\begin{cases}
	e_{i,j}^s\to e_{i,j+1}^s & \text{if } t(i)=0\\
	e_{i,j}^{\nu\mu(s)}\to e_{i,j+1}^{\nu\mu(s)}& \text{if } t(i)=1
	\end{cases}$$
It is easy to see that $h$ defined as above is a simplicial map, that $h(s,0)=s$ and that $h(s,1)=\nu\mu(s)$.
\end{proof}

The main feature of the fibred nerve is that it satisfies the following homotopy preserving property.

\begin{proposition}\label{fiberwise}
Let $\xi=(p:E\to B)$ and $\xi'=(p':E'\to B)$ be fibrations, and let $f:E\to E'$ be a fibred map over $B$.
If $f:E_b\to E'_b$ is a weak equivalence for all objects $b$ of $B$, then $f_*:N_fE\to N_fE'$ is a weak equivalence. 
\end{proposition}

\begin{proof}
By proposition \ref{levelwise} it suffices to prove that the map
$f_*:\{m\mapsto N_f E_{m,n}\}\to\{m\mapsto N_fE'_{m,n}\}$
is a weak equivalence for each $n$.
Faces and degeneracies in direction $m$ preserve the base of a simplex, thus we have decompositions
$$\{m\mapsto N_fE_{m,n}\}=\coprod_{\u b=(b_0\to\dots\to b_n)}N_fE_{\u b}$$
and
$$\{m\mapsto N_fE'_{m,n}\}=\coprod_{\u b=(b_0\to\dots\to b_n)}N_fE'_{\u b}.$$
Moreover, $f_*$ also preserves the base of a simplex, and therefore it can be written as the coproduct of the maps $f_*:N_fE_{\u b}\to N_fE'_{\u b}$.
Now consider the following commutative square.
$$\xymatrix{N_fE_{\u b} \ar[r]^{f_*}\ar[d]^{\mu}& N_fE'_{\u b}\ar[d]^{\mu} \\NE_{b_0} \ar[r]^{f_*}& NE'_{b_0}}$$
The vertical maps are weak equivalences by \ref{m}, and the bottom one is so by hypothesis. It follows from the 3-for-2 property that the upper one is also a weak equivalence and thus the proposition follows.
\end{proof}

\subsection{Cleaved nerve}

The fibration $\ss_{m,n}$ is splitting, since its unique cleavage $\Sigma=\{(id,\alpha)\}$ is closed. We consider $\ss_{m,n}$ as equipped with this cleavage, and we obtain a covariant functor $\ss:\Delta\times\Delta\to \esc\subset\cliv$.

Given $\xi=(p:E\to B)$ a fibration endowed with a cleavage $\Sigma$, we define the {\em cleaved nerve of $(\xi,\Sigma)$} as the bisimplicial set $N_c(\xi,\Sigma)$ whose $m,n$-simplices are given by
$$N_c(\xi,\Sigma)_{m,n}=\cliv(\ss_{m,n},(\xi,\Sigma))$$
We define the {\em cleaved classifying space} of $B_c(\xi,\Sigma)$ as the geometric realization $|d(N_c(\xi,\Sigma))|$ of the cleaved nerve.
These constructions are functorial.
As before, we shall write $N_cE$ and $B_cE$ instead of $N_c(\xi,\Sigma)$ and $B_c(\xi,\Sigma)$ when there is no place to confusion.

The cleaved nerve extends the classical nerve in the sense that there is a natural isomorphism $$d(N_c(\id_B))=d(N_cB)\cong NB,$$
where $\id:B\to B$ is equipped with the cleavage $\Sigma=\fl(B)$.

Note that, if we forget the cleavage $\Sigma$, then we can form the fibred nerve $N_fE$ and there is a natural inclusion in $\bsset$
$$i:N_cE\to N_fE.$$

\begin{lemma}\label{mastil}
Let $\xi=(p:E\to B)$ be a fibration with cleavage $\Sigma$. If $s$ and $s'$ are simplices in $N_cE$ with the same base and the same mast, then $s=s'$. If $\Sigma$ is closed, then for all $\u b\in NB_n$ and $\u a\in (NE_{b_0})_m$ there exists a unique $m,n$-simplex $s\in N_cE$ with base $\u b$ and mast $\u a$.
\end{lemma}

\begin{proof}
Note that $(e_{i,0}^s\to e_{i,1}^s)=(e_{i,0}^{s'}\to e_{i,1}^{s'})$ since they are arrows in $\Sigma$ over $b_0\to b_1$ with the same domain. We see that $(e_{i,j}^s\to e_{i,j+1}^s)=(e_{i,j}^{s'}\to e_{i,j+1}^{s'})$ by iterating this argument. Finally, $(e_{i,j}^s\to e_{i+1,j}^s)=(e_{i,j}^{s'}\to e_{i+1,j}^{s'})$ by the universal property of cartesian arrows. This proves the first assertion.

It is not hard to see that there exists a unique simplex $s\in N_fE$ with base $\u b$, mast $\u a$, and such that $e_{i,j}^s\to e_{i,j+1}^s\in\Sigma$ for all $i,j$. If $\Sigma$ is closed then $e_{i,j}^s\to e_{i,k}^s\in\Sigma$ for all $i,j,k$ and thus $s$ is in $N_cE$ and the second statement holds.
\end{proof}

If the fibration is splitting, then $N_cE$ satisfies a homotopy preserving property analogous to \ref{fiberwise}.

\begin{proposition}\label{cleaved fiberwise}
Let $\xi=(p:E\to B)$ and $\xi'=(p':E'\to B)$ be splitting fibrations with closed cleavages $\Sigma$ and $\Sigma'$, and let $f:E\to E'$ be a cleaved map over $B$.
If $f:E_b\to E'_b$ is a weak equivalence for all object $b$ of $B$ then $f_*:N_cE\to N_cE'$ is a weak equivalence.
\end{proposition}
\begin{proof}
This is analogous to that of \ref{fiberwise}, using the restriction $\mu:N_cE_{\u b}\to NE_{b_0}$, which is also a weak equivalence by \ref{mastil} -- actually, it is an isomorphism.
\end{proof}

The following result asserts that the cleaved nerve suffices to describe the homotopy type of the fibred nerve when the cleavage is closed.

\begin{theorem}\label{fibrado clivado}
If $\xi=(p:E\to B)$ is a splitting fibration with closed cleavage $\Sigma$, then the inclusion $i:N_cE\to N_fE$ is a weak equivalence.
\end{theorem}

\begin{proof}
Again by proposition \ref{levelwise}, we only must show that for each $n$ the inclusion induces a weak equivalence $i_*:\{m\mapsto N_cE_{m,n}\}\to\{m\mapsto N_fE_{m,n}\}$.
For fixed $n$, the map $i_*$ can be written as the coproduct of
$$i_*:N_cE_{\u b}\to N_fE_{\u b}$$
where $\u b$ runs over all $n$-simplices of $NB$. The composition $\mu i_*:N_cE_{\u b}\to NE_{b_0}$ is an isomorphism by \ref{mastil}. It follows by \ref{m} and the 3-for-2-property that $i_*$ is a weak equivalence and thus the proposition.
\end{proof}

If the cleavage $\Sigma$ is not closed, then $N_cE$ and $N_fE$ do not necessarily have the same homotopy type. Let us illustrate this with an example.

\begin{example}
Let $E$ be the category obtained from the ordinal ${\bf\u 3}$ by formally inverting the arrow $2\to 3$. Note that $E$ has an initial element and hence $BE$ is contractible (cf. \ref{contractible}). We shall see $E$ as the total category of a fibration endowed with a cleavage $\Sigma$ in such a way that $N_cE$ is not contractible. Since $d(N_fE)$ and $NE$ have the same homotopy type (see \ref{thm2}), we conclude that in this example the inclusion $i:N_cE\to N_fE$ is not a weak equivalence. 

Let $B={\bf\u 2}$ and let $p:E\to B$ be the surjection which twice takes the value $2$. Clearly it is a fibration. Let $\Sigma$ be the normal cleavage which contains the arrow $0\to3$.
$$\xymatrix@R=15pt{ 0 \ar[r]^{\in\Sigma} \ar@/_/[drr]_{\in\Sigma} & 1 \ar[r]^{\in\Sigma} & 2\ar[d]\\ & & 3\ar[u]}$$
If a simplex $s\in N_cE$ is not contained in the fiber $E_2$, then its mast must be trivial. Since a simplex in $N_cE$ is determined by its mast and its base (cf. \ref{mastil}), it follows that the non-degenerate simplices of $N_cE$ are $0\to1,0\to3,1\to2 \in N_cE_{0,1}$ and some others included in the fiber $E_2$. Thus, the loop  $0\to1\to2\to3\leftarrow0$ gives a non-trivial element of $\pi_1(B_cE,0)$ and therefore $B_cE$ is not contractible.
\end{example}

\section{Relation with the classic nerve}

\subsection{The main result}

Let $\xi=(p:E\to B)$ be a fibration, and let $s=(s_0,s_1)$ be an element of $N_fE_{n,n}$. The composition $s_0\circ diag:{\bf\u n}\to E$
gives a $n$-simplex of $NE$, which we denote by 
$$k(s)=(e_{0,0}^s\to e_{1,1}^s\to\dots\to e_{n,n}^s).$$
This way we get a natural map of simplicial sets $k:d(N_fE)\to NE$
and its geometric realization $k_*:B_fE\to BE$.
We shall see that it is a weak equivalence, so the fibred nerve becomes an alternative model for the homotopy type of $E$.

We prove that $k$ is a weak equivalence first for splitting fibrations and then for any fibration.

\begin{proposition}\label{teo thomason}
Let $\xi=(p:E\to B)$ be a splitting fibration, with closed cleavage $\Sigma$.
Then the map $k|_{d(N_cE)}=ki:d(N_cE)\to NE$ is a weak equivalence.
\end{proposition}

\begin{proof}
(Compare with \cite[1.2]{thomason})
From \ref{very good map} we know that the cleavage $\Sigma$ induces a very good map $s:E^p\to E$ and hence a commutative square
$$\xymatrix{  d(N_c(E^p)) \ar[d]^{s_*}\ar[r]^{ki} & N(E^p)\ar[d]^{s_*} \\d(N_cE) \ar[r]^{ki} &  NE }$$
 by the naturality of $k$.
In this square the vertical arrows are weak equivalences (cf. \ref{good good}, \ref{cleaved fiberwise}), so in order to prove that the bottom arrow is a weak equivalence, by the 3-for-2-property it only remains to show that the upper arrow is one as well.
To do that, we define a map $l:d(N_cE^p)\to NE^p$, prove that there is a simplicial homotopy $ki\cong l$, and prove that $l$ is a weak equivalence.

A simplex $s=(s_0,s_1)$ of $N_cE^p_{m,n}$ is uniquely determined by its mast and its base (cf. \ref{mastil}), so it essentially consists of the following data
$$s=(e_0\to e_1\to \dots\to e_m,p(e_m)\to b_0\to b_1\to\dots\to b_n).$$
For $i=0,\dots,m$, $j=0,\dots,n$, we have $e_{i,j}^s=(e_i,p(e_i)\to b_j)$
, with all the arrows induced by the sequence of above.
Given $i=0,\dots, m$ we define $e_{i,-1}^s$ as the object $(e_i,p(e_i)\to p(e_m))$ of $E^p$ induced by $s$. These new objects lay at the mast of the following simplex of $N_cE^p_{m,n+1}$ induced by $s$.
$$\tilde s=(e_0\to e_1\to\dots\to e_m, p(e_m)\xto{\id} p(e_m)\to b_0\to b_1\to\dots\to b_n)$$
Using $\tilde s$ we define $l:d(N_cE^p)\to NE^p$ by
$$l(s)=(e_{0,-1}^s\to e_{1,-1}^s\to\dots\to e_{n,-1}^s).$$
In the same fashion, the homotopy $h:d(N_cE^p)\times I\to NE^p$ is given by
$$h(s,t)=(e_{0,-1}^s\to\dots\to e_{i-1,-1}^s\to e_{i,i}^s\to\dots\to e_{n,n}^s)$$
where $s\in N_cE^p_{n,n}$, $t\in I_n$, $h(s,t)_j=e_{j,-1}^s$ if $t(j)=0$ and $h(s,t)_j=e_{j,j}^s$ if $t(j)=1$.
One verifies that $h$ is a map, that $h(s,0)=l(s)$ and that $h(s,1)=ki(s)$.

Finally, let us prove that $l$ is a weak equivalence. We regard $NE^p$ as a bisimplicial set constant in direction $n$, so $NE^p_{m,n}=NE^p_m$. The map $l$ is the diagonalization of a bisimplicial map $L:N_cE^p\to NE^p$, defined with the same formula than $l$. The $m$-th component $L_{m,-}$ of $L$ can be identified with the coproduct
$$\coprod_{e_0\to\dots\to e_m}N(p(e_m)/B)\to\coprod_{e_0\to\dots\to e_m}{\rm pt}$$
which is a weak equivalence because $p(e_m)/B$ has an initial element and therefore is contractible (\ref{contractible}). The map $L$ is a weak equivalence by \ref{levelwise} and thus the result.
\end{proof}

\begin{corollary}\label{particular}
If $\xi=(p:E\to B)$ is a splitting fibration, then $k:d(N_fE)\to NE$ is a weak equivalence.
\end{corollary}

\begin{proof}
Fix a closed cleavage $\Sigma$ and then use \ref{fibrado clivado} and \ref{teo thomason}.
\end{proof}

Now we extend \ref{particular} to a non-necessarily splitting fibration. 

\begin{theorem}\label{thm2}
If $\xi=(p:E\to B)$ is a fibration, then the map $k:d(N_fE)\to NE$ is a weak equivalence.
\end{theorem}

\begin{proof}
Let $\Sigma$ be a cleavage of $\xi$. The good map $s:E^p\to E$ induced by $\Sigma$  (cf. \ref{s sigma}) gives a commutative square
$$\xymatrix{  d(N_fE^p) \ar[r]_k \ar[d]_s & NE^p \ar[d]_s \\  d(N_fE) \ar[r]_k & NE}$$
Since the fibration $E^p\to B$ is always splitting, it follows from \ref{particular} that the upper arrow is a weak equivalence.
The vertical arrows are also weak equivalences (cf. \ref{good good}, \ref{fiberwise}) and then the result follows from the 3-for-2-property.
\end{proof}

\begin{example}
The surjection $s:{\bf\u 2}\to{\bf\u 1}$ which takes the value $1$ twice is a fibration.
Down below we show the spaces $B_fE$ and $BE$. The map $k$ is in this case the obvious inclusion.
$$\centerline{
\setlength{\unitlength}{1cm}
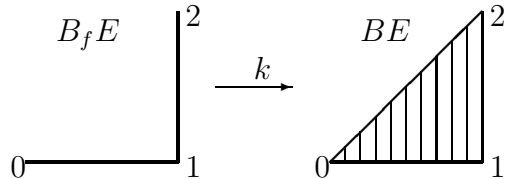
\begin{picture}(6,2)
\thicklines
\put(0,0){\line(1,0)2}
\put(2,0){\line(0,1)2}
\put(4,0){\line(1,0)2}
\put(4,0){\line(1,1)2}
\put(6,0){\line(0,1)2}
\thinlines
\put(2.5,1){\vector(1, 0)1}
\put(4.4,1.6){$BE$}\put(3,1.1){$k$}\put(0.4,1.6){$B_fE$}
\put(-.2,-.2){$0$}\put(2.1,-.2){$1$}\put(2.1,1.8){$2$}
\put(3.8,-.2){$0$}\put(6.1,-.2){$1$}\put(6.1,1.8){$2$}
\put(4.2,0){\line(0,1){.2}}
\put(4.4,0){\line(0,1){.4}}
\put(4.6,0){\line(0,1){.6}}
\put(4.8,0){\line(0,1){.8}}
\put(5,0){\line(0,1){1}}
\put(5.2,0){\line(0,1){1.2}}
\put(5.4,0){\line(0,1){1.4}}
\put(5.6,0){\line(0,1){1.6}}
\put(5.8,0){\line(0,1){1.8}}
\end{picture}}$$
Even when this example is quite simple, it is useful to understand some of the differences between the two constructions. Many of the diagonal arrows in the total category do not provide relevant homotopy information, and the fibred nerve omits them.
\end{example}

The cleaved nerve is smaller than the fibred nerve, and therefore a more effective codification of the homotopy type of the total category. On the other hand, it only works when the fibration is splitting, while the fibred nerve is useful for any fibration.

\subsection{Quillen's Theorem A and its relative version}

Quillen's Theorem A states sufficient conditions for a functor to be a weak equivalence. It was proved to be very useful not only in the work of Quillen but also in many other situations. We derive it here from our framework.

\medskip

The good behaviour of fibred nerve with respect to homotopy (cf. \ref{fiberwise}) together with theorem \ref{thm2} gives the following result.

\begin{proposition}
If $f:E\to E'$ is a fibred map over $B$ such that $f:E_b\to E'_b$ is a weak equivalence for all object $b$ of $B$, then $f$ is a weak equivalence.
\end{proposition}

We deduce both Theorem A and its relative version from this proposition.

\begin{corollary}[Relative Quillen's Theorem A]
Let $u:A\to B$ and $u':A'\to B$ be small categories over $B$. If $f:A\to A'$ is a map over $B$ such that the induced map $A/b\to A'/b$ is a weak equivalence for all $b\in\ob(B)$, then $f$ is a weak equivalence.
\end{corollary}
\begin{proof}
Consider the following commutative square of categories over $B$, where $E^u$ and $E^{u'}$ are the associated fibrations for $u$ and $u'$, and the bottom arrow is induced by $f$ in a natural way.
$$\xymatrix{A \ar[r]^f \ar[d]_i& A'\ar[d]_i \\ E^u \ar[r]^{f_*}& E^{u'}}$$
Since the actual fiber $E^u_b$ identifies with the homotopy fiber $A/b$, the last proposition asserts that the bottom arrow is a weak equivalence, and the vertical ones are also weak equivalences by \ref{replacement}. The result follows from this and the 3-for-2-property.
\end{proof}

\begin{corollary}[Quillen's Theorem A]\label{thm a}
A map $u:A\to B$ between small categories whose homotopy fibers $A/b$ are contractible is a weak equivalence.
\end{corollary}

\begin{proof}
Take $u'=id_B$ in the relative version.
\end{proof}

\subsection{Fibred nerve, cleaved nerve and the codiagonal construction}

In \cite{am} the following construction is introduced.
Given $K$ a bisimplicial set, its {\em codiagonal} (or {\em bar construction}) is the simplicial set $\nabla(K)$ whose $n$-simplices are
$$\nabla(K)_n= \left\{ (x_0,x_1,\dots,x_n) : x_i\in K_{i,n-i},\ d_0^hx_i=d_{i+1}^vx_{i+1}\text{ for }0\leq i<n\right\}$$
and whose faces and degeneracies are
$$d_i(x_0,\dots,x_n)=(d_i^hx_0,d_{i-1}^hx_1,\dots,d_1^hx_{i-1},d_i^vx_{i+1},\dots,d_i^vx_n)$$ and 
$$s_j(x_0,\dots,x_n)=(s_j^hx_0,s_{j-1}^hx_1,\dots,s_0^hx_j,s_j^vx_j,\dots,s_j^vx_n).$$
There is a natural weak equivalence $\theta:d(K)\to\nabla(K)$ defined as follows,
$$\theta(x)=((d_1^v)^n x,(d_2^v)^{n-1}d_0^h x,\dots,(d_{i+1}^v)^{n-i}(d_0^h)^i x,\dots,(d_0^h)^n x)$$
where $x$ is a $n$-simplex of $d(K)$ (cf. \cite{cr}).

In the case of the fibred nerve, both the codiagonal $\nabla N_fE$ and the map $\theta$ can be described in terms of {\em singular functors} of fibrations. 
We shall give these description, and prove that for a splitting fibration there is an isomorphism between the codiagonal of the cleaved nerve and the classic nerve of the total category.

\medskip

Let $T_n$ be the full subcategory of ${\bf\u n}\times{\bf\u n}$ whose objects are the pairs $(i,j)$ satisfying $i\leq j$.
The restriction $pr_2|_{T_n}:T_n\to {\bf\u n}$ is a fibration as one can easily check.
This way we get a covariant functor $T:\Delta\to\fib$, ${\bf\u n}\mapsto T_n$, as the restriction of $\square\circ diag$.

\begin{proposition}\label{codiag1}
Let $\xi=(p:E\to B)$ be a fibration. Then there is a canonical isomorphism of simplicial sets
$$(\nabla N_fE)_n\cong\fib(T_n,\xi)$$
where the right hand side is the {\em singular functor} induced by $T:\Delta\to\fib$.
Under this isomorphisms, the map $\theta$ is identified with the restriction $s\mapsto s|_{T_n}$.
\end{proposition}

\begin{proof}
Let $S$ be the simplicial set $n\mapsto \fib(T_n,\xi)$.
For $k=0,\dots,n$ let $\alpha^k=(\alpha^k_0,\alpha^k_1):\ss_{k,n-k}\to T_n$ be the fibred map satisfying $\alpha^k_0(i,j)=(i,j+k)$ for all $(i,j)\in\ob(\ss_{k,n-k})$.
We define $\lambda:S\to\nabla N_fE$ by mapping an $n$-simplex $x:T_n\to\xi$ to
$\lambda(x)=(x\alpha^0,x\alpha^1,\dots,x\alpha^n)$.
It is straightforward to check that $\lambda$ is well defined, i.e. the coordinates of $\lambda(x)$ satisfy the compatibility conditions of the codiagonal, and that $\lambda$ respects the faces and degeneracies.

To see that $\lambda$ is actually an isomorphism, we remark that a simplex 
$s\in N_fE_{m,n}$ can be presented as an array of $m\times n$ commutative little squares of $E$
$$\xymatrix@R=20pt@C=10pt{e_{i,j}^s \ar[r] \ar[d] & e_{i,j+1}^s\ar[d]\\
			e_{i+1,j}^s \ar[r] & e_{i+1,j+1}^s}$$
on which the vertical arrows are over identities and the horizontal ones are cartesian arrows.
If $y_k\in N_fE_{k,n-k}$, $k=0,\dots,n$, the equation $d_0^h y_k=d_{i+1}^v y_{k+1}$ says that the array of $k\times (n-k-1)$ little squares obtained from $y_k$ by deleting the first column equals the array obtained from $y_{k+1}$ by deleting the last row.



It is clear from these descriptions that a simplex $x\in S_n$ identifies with a sequence $(y_0,\dots,y_n)$, $y_k\in N_fE_{k,n-k}$, under the compatibility conditions that impose the codiagonal.
%
%
%
%
\end{proof}

A similar statement holds for the cleaved nerve. Its proof is essentially that of \ref{codiag1}.
Note that $T_n$ inherits the closed cleavage from $\ss_{n,n}$ and hence $T$ can be considered as a functor $\Delta\to\esc\subset\cliv$.

\begin{proposition}\label{codiag2}
Let $\xi=(p:E\to B)$ be a fibration with cleavage $\Sigma$. Then there is a canonical isomorphism of simplicial sets
$$(\nabla N_cE)_n\cong\cliv(T_n,(\xi,\Sigma))$$
where the right hand side is the {\em singular functor} induced by $T:\Delta\to\esc\subset\cliv$.
Under this isomorphisms, the map $\theta$ identifies with the restriction $s\mapsto s|_{T_n}$.
\end{proposition}

Given $\xi=(p:E\to B)$ a fibration endowed with a cleavage $\Sigma$, we have the following diagram of simplicial sets, where $\bar k$ is defined below.
$$\xymatrix@R=5pt{
dN_cE \ar[r]^i \ar[dd]^\theta & dN_fE \ar[dr]^k \ar[dd]^\theta &  \\
 &  & NE \\
\nabla N_cE \ar[r]^i & \nabla N_fE  \ar[ru]^{\bar k} & }$$
If $x=(x_0,x_1)\in\fib(T_n,E)$, then the composition $x_0\circ diag:{\bf\u n}\to E$ defines a simplex in $NE_n$. Under the identification of \ref{codiag1} this gives the map $\bar k:\nabla N_fE\to NE$, $\bar k(x)=x_0\circ diag$.

The following shows how to recover the classic nerve of the total category of a splitting fibration from the cleaved nerve.

\begin{theorem}\label{codiag3}
If the cleavage $\Sigma$ is closed, then the map $\bar k i:\nabla N_cE\to NE$ is an isomorphism.
\end{theorem}

\begin{proof}
The proof is similar to that of lemma \ref{mastil}. To see that $\bar k i$ is injective, consider a simplex $x\in NE_n$, view $x:{\bf\u n}\to E$ as defined over the diagonal of $\ss_{n,n}$ and note that an extension $s:T_n\to E$ of $x$ is necessarily unique: The horizontal arrows must belong to the cleavage, and the vertical ones are uniquely determined by the universal property of cartesian arrows.

If $\Sigma$ is closed, then the unique functor $s:T_n\to E$ such that $s \circ diag = x$ and $e_{i,j}^s\to e_{i,j+1}^s\in\Sigma$ determines a cleaved map $T_n\to E$ and hence a simplex $s\in(\nabla N_cE)_n$ satisfying $\bar k i(s)=x$. Thus the surjectivity.
\end{proof}

\section{Other examples and applications}

\subsection{Homotopy colimits}
\label{Grothendieck construction}


Bousfield and Kan \cite{bk} give a construction of a representing object for the homotopy colimit of a diagram of simplicial sets. Given $Z:I\to\sset$, let $hc(Z)$ be the bisimplicial set whose $m,n$-simplices are
$$hc(Z)_{m,n}=\coprod_{i_0\to\dots\to i_n} Z(i_0)_m$$
where the coproduct runs over all simplices of dimension $n$ of $NI$, and faces and degeneracies are defined in the obvious way. Then $hc(Z)$ satisfies the homotopy universal property of homotopy colimits (cf. \cite{bk}).

In \cite{thomason} Thomason uses the Bousfield-Kan construction to describe homotopy colimits in $\cat$ in terms of the Grothendieck construction for a functor.
We recall Grothendieck construction over a functor $F:B\to\cat$, compare the Bousfield-Kan construction with the cleaved nerve and derive Thomason's theorem from this. 

\medskip

Given $F:B\to\cat$ a diagram of small categories, its {\em Grothendieck construction} is a splitting fibration $F\rtimes B\to B$ whose fibers are the values of $F$. The objects of the total category $F\rtimes B$ are pairs $(x,b)$ with $b$ an object of $B$ and $x$ and object of $F(b)$. An arrow $(f,\phi):(x,b)\to(x',b')$ in $F\rtimes B$ is a pair $\phi:b\to b'$, $f:F(\phi)(e)\to e'$. Composition is given by 
$(\psi,g)\circ(\phi,f)=(\psi\phi,gF(\psi)(f))
$. The map $F\rtimes B\to B$ is the projection, and the arrows $(\id,\phi)$ form a distinguished closed cleavage.

\begin{theorem}\label{hc}
Given $F:B\to \cat$, there is an isomorphism $N_c(F\rtimes B)\xto\sim hc(NF)$ between the cleaved nerve of the Grothendieck construction of $F$ and the Bousfield-Kan construction for homotopy colimits.
\end{theorem}
\begin{proof}
The isomorphism $N_c(F\rtimes B)\to hc(NF)$ maps a $m,n$-simplex $s$ of $N_c(F\rtimes B)$ to the element $\u a$ in the summand indexed by $\u b$, where $\u a$ is the mast of $s$ and $\u b$ is the base of $s$. This is indeed a morphism of bisimplicial sets, and it is invertible because of \ref{mastil}.
\end{proof}

If $E$ is any fibration, then one can define a function $N_f E\to hc(NF)$ in a similar fashion. However, this function is not a morphism in general, as it does not respect the 0-th face operator.

\begin{corollary}[Thomason's theorem]\label{thm t}
The Grothendieck construction $F\rtimes B$ over a functor $F:B\to\cat$ is a representing object for the homotopy colimit of $F$.
\end{corollary}
\begin{proof}
This is a consequence of \ref{teo thomason} and \ref{hc}.
\end{proof}

\subsection{Spectral sequence of a fibration}

A bisimplicial set gives rise to a bisimplicial abelian group and hence to a bicomplex. In this section we study the spectral sequence associated to the bicomplex coming from the fibred nerve.
We recall some definitions on homology of categories from \cite{quillen}. Then we describe how a fibration gives rise to a pseudofunctor, and define the modules $H_m(F)$. Finally we state and prove theorem \ref{spectral sequence} and derive a homology version of Quillen's Theorem A as a corollary.

\medskip

Given a small category $C$, a {\em module over $C$} is a functor $A:C\to\ab$, where $\ab$ denotes the category of abelian groups.
The $m$-th {\em homology group of $C$ with coefficients in a module $A$} is defined as the $m$-th left derived functor of $\colim:\ab^C\to\ab$.
$$H_m(C,A)=\colim^m_C A$$ 
The groups $H_m(C,A)$ can be computed as the homology of the following simplicial abelian group
$$C_m(C,A)=\bigoplus_{c_0\to\dots\to c_m} A(c_0)$$
and, in the case that $A$ is morphism inverting, they agree with the homology of the classifying space $BC$ with local coefficients induced by $A$.
We write $H_m(C)$ instead of $H_m(C,A)$ when $A$ is the constant functor $\Z$. It follows that
$$H_m(C)=H_m(BC)$$
where the right side denotes the singular homology of the space $BC$.

\medskip

Let $\xi=(p:E\to B)$ be a fibration, and let $\Sigma$ be a cleavage of $\xi$. 
For each arrow $\phi:b\to b'$ in $B$ a {\em base-change functor} $\phi_*:E_b\to E_{b'}$ is defined as follows: If $e$ is an object of $E_b$, then $\phi_*(e)$ is the codomain of $\Sigma_{e,\phi}$, and if $f:e\to e'$ is an arrow of $E_b$, then $\phi_*(f)$ is the unique arrow in $E_{b'}$ such that $\phi_*(f)\circ \Sigma_{e,\phi}=\Sigma_{e',\phi}\circ f$.

Of course, $\phi_*$ depends on the cleavage, but different cleavages give rise to naturally isomorphic base-change functors, as follows from the universal property of cartesian arrows.
In the same fashion, given $\phi$, $\psi$ composable arrows of $B$, there is a natural isomorphism $\psi_*\phi_*\then(\psi\phi)_*$. The set of data
$$b\mapsto E_b \qquad \phi\mapsto \phi_* \qquad \psi_*\phi_* \then (\psi\phi)_*$$
defines a {\em pseudofunctor} $B\dasharrow\cat$ (cf. \cite{sga1}). Note that if $\Sigma$ is closed then the isomorphisms $\psi_*\phi_*\then(\psi\phi)_*$ are identities and one has a true functor $B\to\cat$.

Given $m\geq 0$, let $H_m(F):B\to\ab$ be the functor which assigns to each $b\in B$ the group $H_m(E_b)$, and to each arrow $\phi:b\to b'$ the map induced by $\phi_*$. Since isomorphic functors yields homotopic maps between the classifying spaces, the module $H_m(F)$ is well defined (i.e. is a functor) and does not depend on the cleavage $\Sigma$.

\begin{theorem}\label{spectral sequence}
There is a spectral sequence $\{X_{m,n}^r\}$ which converges to the homology of the total category $E$ and whose second sheet consists of the homology of the base with coefficients in the homology of the fibers.
$$X^2_{m,n}=H_n(B,H_m(F))\then H_{m+n}(E)$$
\end{theorem}
\begin{proof}
From the bisimplicial set $N_fE$ we construct the free bisimplicial abelian group $\Z N_fE$, and the bicomplex $C_fE$, whose $m,n$-th group equals that of $\Z N_fE$ and whose horizontal and vertical differential maps are the alternate sum of the horizontal and vertical faces, respectively.
We have that
$$C_fE_{m,n}=\bigoplus_{\u b=(b_0\to\dots\to b_n)}\Z[(N_fE_{\u b})_m]$$
where $(N_fE_{\u b})_m$ is the set of $m,n$-simplices of $N_fE$ with base $\u b$. Filtering the bicomplex $C_fE$ in the horizontal direction gives a spectral sequence
$$H_n(H_m(C_fE))\then H_{m+n}(Tot(C_fE)).$$
The first sheet of this spectral sequence is obtained by computing the vertical homology ($m$-direction) of $C_fE$.
In degree $m,n$ this is equal to
$$H_m(C_fE)_{m,n}=\bigoplus_{\u b} H_m(N_fE_{\u b})\cong\bigoplus_{\u b} H_m(E_{b_0}),$$
where the isomorphism $\cong$ is that induced by $\mu$ (cf. \ref{m}).
The second sheet of this spectral sequence is obtained by computing the horizontal homology ($n$-direction). In degree $m,n$ this is equal to
$H_nH_m(C_fE)=H_n(B,H_m(F))$.
Finally, by the generalized Eilenberg-Zilber theorem (cf. \cite[IV,2.5]{gj}) the homology of the total complex $H_{m+n}(Tot(C_fE))$ is isomorphic to the homology of the diagonal $H_{m+n}(d(\Z N_fE))$, which equals $H_{m+n}(E)$ since $d(N_fE)$ and $NE$ are homotopic (cf. \ref{thm2}). This completes the proof.
\end{proof}

Suppose now that $\xi=(p:E\to B)$ is a fibration whose fibers are homologically trivial, namely
$H_m(E_b)=0$ if $m>0$ and $H_0(E_b)=\Z$ for all objects $b$ of $B$.
If $m>0$ then the functors $H_m(F)$ are constant and equal to $0$, so the second sheet of the spectral sequence $X$ of \ref{spectral sequence} is 
$$X^2_{m,n}=\begin{cases}0 & \text{if }m>0\\ H_n(B) & \text{if }m=0\end{cases}$$
It follows that $X^\infty=X^2$ and thus the homology of $E$ is that of $B$. 
It is not hard to see that $p_*:H_n(E)\to H_n(B)$ is actually the isomorphism.

\begin{corollary}\label{homology}
If a fibration $\xi=(p:E\to B)$ is such that the fibers $E_b$ are homologically trivial, then $p_*:H_n(E)\xto\sim H_n(B)$ is an isomorphism for all $n\geq 0$.
\end{corollary}

\begin{corollary}[Homology version of Quillen's Theorem A]\label{homology2}
Let $u:A\to B$ be a map between small categories whose homotopy fibers $A/b$ are homologically trivial. Then $u_*:H_n(A)\xto\sim H_n(B)$ is an isomorphism for all $n\geq 0$.
\end{corollary}
\begin{proof}
If $\pi:E^u\to B$ is the fibration associated to $u$, then its fibers are isomorphic to the homotopy fibers of $u$ and thus $\pi$ induces isomorphisms in homology by \ref{homology}. Since $u=\pi\circ i$ and $i:A\to E^u$ is a weak equivalence (cf. \ref{replacement}) the result follows.
\end{proof}

\subsection{Quillen fibrations and Theorem B}

We have seen in many examples that different fibers of a Grothendieck fibration need not have the same homotopy type. 
This remark shows that in general the map $B_fE\to B_fB$ is not a fibration, nor a quasifibration.
It is remarkable that this is the only obstruction. 
In this section we define Quillen fibrations, discuss the monodromy action and reformulate Quillen's Theorem B in terms of the fibred nerve.

\medskip

We say that a fibration $\xi=(p:E\to B)$ is a {\em Quillen fibration} if for each arrow $\phi:b\to b'$ in $B$ the base-change functor $\phi_*:E_b\to E_{b'}$ is a weak equivalence. Note that this definition does not depend on the cleavage, for two base-change functors over $\phi$ must be homotopic.

In a Quillen fibration the induced functor $B\to[\top]$, $b\mapsto BE_b$ is morphism inverting, therefore it induces a map $$\pi_1(B)\to[\top].$$
Here $\pi_1(B)$ denotes the fundamental groupoid of $B$, namely the groupoid obtained by formally inverting all the arrows of $B$, and $[\top]$ denotes the category of topological spaces and homotopy classes of continuous maps. We call $\pi_1(B)\to[\top]$ the {\em monodromy action} of the fibration.
The monodromy action is a first tool to classify Quillen fibrations. In very special situations, it suffices to recover the whole fibration, as we can see in the following example.

\begin{example}
(cf. \cite{quillen}) If $p:E\to B$ is a Quillen fibration with discrete fibers (the only arrows in the fibers are identities), then the base-change functors $E_b\to E_{b'}$ must be bijections. It follows that a Quillen fibration with discrete fibers is essentially the same as a functor $B\to\set\subset\cat$ which is morphism inverting, or what is the same, a functor $\pi_1(B)\to\set$.

A Quillen fibration $p:E\to B$ with discrete fibers should be thought of as a covering of categories. Indeed, they yield coverings after applying the classifying space functor.
\end{example}

One is interest in understand how fibrations behave with respect to the classifying space functor. The next example shows that $B_fE\to B_fB$ need not be a fibration, so we shall look for a notion weaker than that.

\begin{example}
Let $E$ be the full subcategory of $I\times I$ with objects $(1,0)$, $(0,1)$ and $(1,1)$. Then the second projection $E\to I$ is a fibration. Since the fibers are contractible it is, in fact, a Quillen fibration. Despite this, the induced map of topological spaces is not a fibration, as one can easily check.
\end{example}

Recall that a {\em quasifibration} of topological spaces $f:X\to Y$ is a map such that the inclusions of the actual fibers into the homotopy fibers are weak homotopy equivalences. They extend the notion of fibration, and their most important feature is that they yield long exact sequences relating the homotopy groups of the fibre, the total space and the base space.


\begin{theorem}\label{prop}
If $p:E\to B$ is a Quillen fibration, then the induced map $p_*:B_fE\to B_fB$ is a quasifibration of topological spaces.
\end{theorem}

\begin{proof}
It is essentially that of \cite[lemma p.14]{quillen}.
We endow $B_fB=BB$ with the canonical cellular structure. We prove that the restriction of $p_*$ to the $n$-th skeleton $sk_n(BB)$ is a quasifibration by induction on $n$, from which the result follows. To prove the inductive step we write $sk_n(BB)$ as the union $U\cup V$, where $U$ is obtained by removing the barycenters of the $n$-cells and $V$ is the union of the interiors of the $n$-cells, and prove that $p_*$ is a quasifibration when restricted to $U$, $V$ and $U\cap V$.

We denote $|N_fE_{\u b}|$ by $B_fE_{\u b}$ (see \ref{ultima}). Realizing first in the $m$-direction, 
the restriction of $p_*$ to the interior of the $n$-cell indexed by $\u b\in NB_n$ can be identified with the restriction of the projection $B_fE_{\u b}\times \Delta^n\to\Delta^n$ to the interior of the topological $n$-simplex $\Delta^n$. It follows from this that $p_*|_V$ and $p_*|_{U\cap V}$ are quasifibrations.

We deform $p_*^{-1}(U)$ into $p_*^{-1}(sk_{n-1}(B_fB))$ by using the radial deformation of $\Delta^n$ minus its barycenter into $\partial\Delta^n$, and use the inductive assumption to conclude that $p_*|_U$ is a quasifibration.
We must verify that if this deformation carries $x$ into $x'$, then the map $g:p_*^{-1}(x)\to p_*^{-1}(x')$ induced by the deformation is a weak homotopy equivalence. Let $x$ be a point in the interior of the $n$-cell indexed by $\u b\in NB_n$. If the radial deformation push $x$ into the open cell indexed by the face $\u b'$ of $\u b$, then $p_*^{-1}(x)=B_fE_{\u b}$ and $p_*^{-1}(x')=B_fE_{\u b'}$. Fixed a cleavage $\Sigma$, the composition (cf. \ref{m})
$$BE_{b_0}\xto\nu B_fE_{\u b}\xto g B_fE_{\u b'} \xto\mu BE_{b_0'},$$ 
(with $\nu=\nu(\u b)$ and $\mu=\mu(\u b')$) equals a base-change functor over the arrow $b_0\to b'_0$ of $\u b$ (more precisely, it is the composition of the base-changes given by $\Sigma$ over the arrows $b_i\to b_{i+1}$). Since $p$ is a Quillen fibration, and $\nu$ and $\mu$ are weak equivalences, it follows from the 3-for-2-property that $g$ is a homotopy equivalence and thus the result.
\end{proof}

The last theorem shows an interesting feature of the fibred nerve: it carries Quillen fibrations into quasifibrations. The question of whether or not $BE\to BB$ is a quasifibration is rather unclear, and this can be understood as a disadvantage of the classic nerve when dealing with fibrations.

\begin{corollary}[Quillen's Theorem B]\label{thm b}
If $u:A\to B$ is a map between small categories such that $A/b\to A/b'$ is a weak equivalence for all $b\to b'\in\fl(B)$, then there is a long exact sequence 
$$\dots\xto\partial \pi_k(A/b,\overline a)\to\pi_k(A,a)\xto{u_*}\pi_k(B,b)\xto\partial\pi_{k-1}(A/b,\overline a)\to\dots$$
where $a\in\ob(A)$, $b=u(a)$ and $\overline a=(a,\id_b)$.
\end{corollary}
\begin{proof}
Let $i:A\to E^u$ be the canonical map into the associated fibration, $r$ its right adjoint and $w:A/b\to A$ be the map $(a,u(a)\to b)\mapsto a$. In the diagram
$$\xymatrix{A/b \ar[r]^w\ar@{=}[d] & A \ar[r]^u & B\ar@{=}[d]\\
		E^u_b \ar@{^{(}->}[r] & E^u \ar[r]^\pi \ar[u]^r & B}$$
the left square commutes, and since $\pi i=u$ and $r$ is a homotopy inverse to $i$ the right square commutes up to homotopy.
We conclude that the homotopy groups of $A/b$, $A$ and $B$ can be identified naturally with that of the base, the fiber and the total category of the associated fibration $E^u\to B$. It is a Quillen fibration, and thus the result follows from \ref{prop}.
\end{proof}

\subsection{Group actions and TCP}

Regarding small categories as combinatorial models for homotopy types, it is natural to investigate how they behave under the action of a group.
In this section we derive a splitting fibration from a small category endowed with a group action, and relate its cleaved nerve with a twisted cartesian product in the sense of \cite{may}. We also study the spectral sequence \ref{spectral sequence} in this particular case.

\medskip

A simplicial group $G$ {\em operates on a simplicial set $K$ (from the left)} if there is a simplicial map $G\times K\to K$, $(g,k)\mapsto g\cdot k$ satisfying 
$1_n\cdot k=k$ for all $k\in K_n$ and $g_1\cdot(g_2\cdot k)=(g_1g_2)\cdot k$ for all $k\in F_n$ and $g_1,g_2\in G_n$. Here $1_n$ denotes the unit of $G_n$.
Given $A,B$ simplicial sets and $G$ a simplicial group which operates on $A$, a {\em twisted cartesian product (TCP)} with fibre $A$, base $B$ and group $G$ is a simplicial set $A\times_\tau B$ with simplices $(A\times_\tau B)_n=A_n\times B_n$ and faces and degeneracies given by
$$d_i(a,b)=\begin{cases}(d_ia,d_ib) & i>0\\
						(\tau(b)\cdot d_0a,d_0b) & i=0\end{cases}\qquad
	s_i(a,b)=(s_ia,s_ib),\ \ i\geq0.$$
Here $\tau:B_n\to G_{n-1}$ is a function which must satisfy some standard identities in order to make $A\times_\tau B$ a simplicial set. This $\tau$ is called the {\em twisting function}.

Let $G$ be a group, and let $A$ be a small category on which $G$ acts. This action can be seen as a group morphism $G\to Aut(A)$, $g\mapsto u_g$, or equivalently, as a functor $G\to\cat$ that maps the unique object of $G$ to $A\in\ob(\cat)$. The Grothendieck construction over this functor is a splitting fibration $p:G\rtimes A \to G$ over $G$.

The constant simplicial group $G$ (in which every face and degeneracy operator is the identity) operates on $NA$ from the left via the formula
$$g\cdot (a_0\to\dots\to a_n) =(u_g(a_0)\to \dots\to u_g(a_n)).$$

\begin{proposition}
The diagonal of the cleaved nerve $dN_c(G\rtimes A)$ can be regarded as a TCP between the nerves of $A$ and $G$, namely $NA\times_\tau NG$.
\end{proposition}
\begin{proof}
Let $\tau:NG_n\to G_{n-1}=G$ be the projection $\tau(\ast\xto{g_1}\ast\xto{g_2}\dots\xto{g_n}\ast)=g_1$, and let $NA\times_\tau NG$ be the TCP with twisting function $\tau$. We define a simplicial map 
$\varphi:dN_c(G\rtimes A)\to NA\times_\tau NG$ by giving to each simplex $s\in N_c(G\rtimes A)_{n,n}$ the pair $(a,b)$, where $a$ is the mast of $s$ and $b$ is its base. One checks easily that $\varphi$ is actually a simplicial map, and it is an isomorphism because $G\rtimes A\to G$ is splitting, together with \ref{mastil}.
\end{proof}

Given $G$ acting on $A$, the fibration $G\rtimes A\to G$ has a unique fiber, which is isomorphic to $A$. Thus the modules $H_m(F)$ of the spectral sequence \ref{spectral sequence} are just the homology groups of $A$ endowed with the action of $G$. Writing $A/\!/G=G\rtimes A$ for the homotopy theoretic quotient (cf. \ref{thm t}) we obtain the following version of the Eilenberg-Moore spectral sequence (cf. \cite[p.775]{anderson}) as an application of \ref{spectral sequence}.

\begin{proposition}
There is a spectral sequence $\{X_{m,n}^r\}$ which converges to the homology of the homotopy theoretic quotient $A/\!/G$ and whose second sheet consists of the group homology of $G$ with coeficients in the homology of $A$.
$$X^2_{m,n}=H_n(G,H_m(A))\then H_{m+n}(A/\!/G)$$
\end{proposition}

\bigskip

\bigskip


 \texttt{mdelhoyo@dm.uba.ar}

\end{document}